\documentclass{amsart}

\usepackage{color} 
\usepackage[automark]{scrpage2} 
\usepackage{mathrsfs,amsmath,marvosym,amsthm,amssymb} 
\usepackage[T1]{fontenc} 
\usepackage{paralist} 
\usepackage{nicefrac} 
\usepackage[arrow, matrix, curve]{xy} 
\usepackage{extpfeil} 
\usepackage{enumitem}
\usepackage[utf8]{inputenc}
\usepackage{color}

\usepackage{url}

\usepackage{bbm}
\usepackage{setspace}

\usepackage{mathtools}

\usepackage{multibib}

\newcites{eigenes}{Eigene Publikationen}

\newcommand{\Q}{\mathbb{Q}}
\newcommand{\Z}{\mathbb{Z}}
\newcommand{\PP}{\mathbb{P}}

\newcommand{\F}{\mathscr{F}}

\newcommand{\Gm}{\mathbb{G}_m}
\newcommand{\Ga}{\mathbb{G}_a}

\newcommand{\OO}{\mathcal{O}}
\newcommand{\G}{\mathscr{G}}

\newcommand{\LL}{\mathscr{L}}

\newcommand{\A}{\mathbb{A}}

\newcommand{\K}{\mathscr{K}}

\newcommand{\E}{\mathscr{E}}
\newcommand{\HH}{\mathscr{H}}
\newcommand{\EE}{\mathscr{E}}
\newcommand{\HOM}{\mathscr{H}om}

\newcommand{\END}{\mathscr{E}nd}

\newcommand{\Spec}{\textup{Spec}}

\newcommand{\KK}{\mathbb{C}(\!(t)\!)}

\newcommand{\Aut}{\textup{Aut}}

\newcommand{\rk}{\textup{rk}}

\newcommand{\rig}{\textup{rig\,}}

\newcommand{\GL}{\textup{GL}}
\newcommand{\Gal}{\textup{Gal}}
\newcommand{\MC}{\textup{MC}}

\newcommand{\piet}{\pi_1^{\textup{ét}}}
\newcommand{\specialcell}[2][c]{%
  \begin{tabular}[#1]{@{}c@{}}#2\end{tabular}}
\newcommand{\Ql}{\Q_\ell}
\newcommand{\Qlb}{{\overline{\Q}_\ell}}

\newcommand{\Perv}{\textup{Perv}}
\newcommand{\pr}{\textup{pr}}
\newcommand{\Rb}{\mathbf{R}}
\newcommand{\Sw}{\textup{Sw}}

\newcommand{\midd}{{\textup{mid}}}

\newcommand{\FI}{\F^{(0,\infty')}}

\newcommand{\U}{\mathbf{U}}

\newcommand{\IF}{\mathbb{F}}
\newcommand{\Itame}{I^{\textrm{tame}}}
\newcommand{\Ind}{\textrm{Ind}}
\newcommand{\Res}{\textrm{Res}}
\newcommand{\one}{\mathbbm{1}}

\newcommand{\Ktame}{K^{\textup{tame}}}

\theoremstyle{definition}
\newtheorem{defi}{Definition}[section]

\theoremstyle{plain}

\newtheorem{cor}[defi]{Corollary}
\newtheorem{thm}[defi]{Theorem}
\newtheorem{lem}[defi]{Lemma}
\newtheorem{prop}[defi]{Proposition}
\theoremstyle{remark}

\title{Wildly Ramified Rigid $G_2$-Local Systems}
\author{Konstantin Jakob}
\address{Massachusetts Institute of Technology, Department of Mathematics, 77 Massachusetts Ave, MA 02139}
\subjclass[2010]{20G41, 14F05}

\begin{document}
 
\maketitle

\begin{abstract}
In earlier work of the author rigid irregular connections with differential Galois group $G_2$ and whose slopes have numerator $1$ were classified and new rigid connections were constructed. The same construction can be carried out for $\ell$-adic local systems in the setting of positive characteristic. In this article we provide the results that are needed to obtain the classification of wildly ramified rigid $G_2$-local systems whose slopes have numerator $1$. The overall strategy of the classification is very similar but the methods needed to obtain some invariants differ. 
\end{abstract}

\section{Introduction}
Rigid local systems are local systems which are globally determined by their local monodromy. They have been studied in detail by Katz in \cite{Ka96} who proved that any such local system arises from a system of rank one by iterating tensor products with rank one local systems and middle convolution. To include equations or connections with irregular singularity, Arinkin has extended the result of Katz by additionally involving Fourier-Laplace transform of $D$-modules in \cite{Arinkin10}. This builds on work of Bloch and Esnault who prove in \cite{BlochEsnault04} that Fourier-Laplace transform preserves rigidity. The statement is that any rigid irreducible connection (with possibly irregular singularities) can be obtained from a connection of rank one by iterating Fourier-Laplace transforms. In the article \cite{Ja20}, this method of construction was used to give a classification of rigid irregular irreducible connections with differential Galois group $G_2$ and whose slopes have numerator $1$. \par
When working with $\ell$-adic sheaves on some open subset $U\subset \PP^1_k$ over the algebraic closure $k$ of a finite field $\IF_q$ of characteristic $p$ one can prove similar results. There are a lot of similarities and analogies in both settings, but unfortunately not everything translates directly from one to the other. The goal of this article is to introduce the necessary tools and methods to transfer the classification of \cite{Ja20} to the arithmetic setting.  \par
Let us explain the strategy of the classification. Rigid local systems can be identified through a cohomological invariant. An irreducible $\ell$-adic local system $\LL$ on $U\subset \PP^1$ is rigid if and only if 
\[\chi(\PP^1,j_*\END(\LL))=2\]
where $j_*$ denotes the non-derived direct image along the open embedding $j:U\hookrightarrow \PP^1$. For this reason we will call 
\[\rig(\LL):=\chi(\PP^1,j_*\END(\LL)\]
the index of rigidity. The fact that $\rig(\LL)=2$ implies rigidity of $\LL$ is essentially a consequence of Poincaré duality for $\ell$-adic sheaves. The other direction is more complicated and was recently proven by Fu in \cite{Fu19} using rigid analytic geometry. Let $S=\PP^1-U$. Using the Euler-Poincaré formula, one can compute the index of rigidity through local invariants as follows
\[\rig(\LL)=(2-|S|)\rk(\LL)^2+\sum_{s\in S} \Sw_s(\LL)+\dim(\LL)^{I_s}.\]
One of the main ingredients of the classification in \cite{Ja20} is a classic result of Levelt-Turittin for formal connections which allows to decompose any such connection into a direct sum of objects of the form 
\[[r]_*(\EE^\varphi \otimes R) \]
where $[r]$ denotes an $r$-fold covering of the formal punctured disc, $\EE^\varphi=(\KK, d+d\varphi)$ is a formal connection with an exponential solution and $R$ is some regular singular formal connection. For objects of this form one knows how to compute the invariants needed to compute the index of rigidity. We will see that the same is true for representations of the inertia group $I=\Gal(k(\!(t)\!)^{sep}|k(\!(t)\!))$ corresponding to sheaves of the form 
\[[r]_*(\LL_\psi(\varphi)\otimes \K) \]
where $\LL_\psi$ is the restriction of an Artin-Schreier sheaf (for some fixed non-trivial additive character $\psi:\IF_p\rightarrow \Qlb^*$), $\LL_\psi(\varphi)$ denotes pull-back of $\LL_\psi$ by the morphism given by the polynomial $\varphi\in t^{-1}k[t^{-1}]$ and $\K$ some tamely ramified sheaf on the punctured formal disc. In general, an irreducible representation of $I$ might not be of the above form, i.e. an analogue of the Levelt-Turittin decomposition does not exist in positive characteristic. There is however a weaker form proven by Fu in \cite{Fu10}. In the same article he raised the following question. Given an irreducible continuous $\Qlb$-representation $V$ of $I$, does there exist a tame character $\chi:I\rightarrow \Qlb^*$ such that $\chi\otimes V$ has finite image? \par
We answer this question positively, strengthening his result \cite[Proposition 0.5]{Fu10} to the following statement. 
\begin{prop} Let $\rho: I\rightarrow \GL(V)$ be an indecomposable continuous $\Qlb$-representation and denote by $P$ the wild ramification subgroup of $I$. Suppose that $\rho(P^p [P,P])=1$ and that the Swan conductor $\Sw(V)<p$. In this case, $V$ is isomorphic to the representation corresponding to $[r]_*(\LL_\psi(\varphi)\otimes \K)$
for an integer $r$ prime to $p$, $\varphi\in t^{-1}k[t^{-1}]$ and $\K$ some tamely ramified sheaf. 
\end{prop}
In our setting this result suffices to conclude that the local monodromy of the rigid systems we will consider decomposes into these simple objects. We will compute tensor products and determinants of such representations and attach to them invariants which are similar to formal monodromy and exponential torus of a formal connection (these are invariants coming from differential Galois theory). This will in turn allow us to conclude the following classification theorem which is a generalization of the classification of tame rigid $G_2$-local systems by Dettweiler and Reiter in \cite{Dett10}. 
\begin{thm}\label{classloc} Let $k$ be the algebraic closure of a finite field of characteristic $p>7$. Let $\lambda_1,\lambda_2\in k$ such that $\lambda_1\neq \pm\lambda_2$ and let 
\[\chi, x, y, z, \varepsilon,\iota :\varprojlim_{(N,p)=1}\mu_N(k)\rightarrow \Qlb \]
be non-trivial characters such that $\chi$ is not quadratic, $z^4$ is non-trivial, $x,y,xy$ and their inverses are pairwise different and such that $\varepsilon$ is of order $3$ and $\iota$ is of order $4$. Denote by $\overline{\chi}$ the inverse of $\chi$, by $\one$ the trivial representation of rank one and by $-\one$ the unique character of order $2$. Every pair of local monodromies in the following list is exhibited by some irreducible rigid $\ell$-adic local system of rank $7$ on $\Gm$ with monodromy group $G_2(\Qlb)$.
\begin{center}
\begin{tabular}{ c c }
$0$ & $\infty$ \\
\hline \\
$\U(3)\oplus \U(3)\oplus \mathbbm{1}$ & \specialcell[c]{$[2]_*(\LL_\psi(\lambda_1 u^{-1})\otimes (\chi\oplus\overline{\chi}))$ \\ $\oplus\, [2]_*(\LL_\psi(2\lambda_1 u^{-1})) \oplus (-\mathbbm{1})$} \\ [15pt]
$-\U(2)\oplus -\U(2)\oplus \mathbbm{1}^{3} $ & \specialcell[c]{$[2]_*(\LL_\psi(\lambda_1 u^{-1})\otimes (\chi\oplus\overline{\chi}))$ \\ $\oplus\, [2]_*(\LL_\psi(2\lambda_1 u^{-1})) \oplus (-\mathbbm{1})$} \\ [15pt]
$x\oplus x\oplus\overline{x}\oplus\overline{x}\oplus\mathbbm{1}^3$ & \specialcell[c]{$[2]_*(\LL_\psi(\lambda_1 u^{-1})\otimes (\chi\oplus\overline{\chi}))$ \\ $\oplus\, [2]_*(\LL_\psi(2\lambda_1 u^{-1})) \oplus (-\mathbbm{1})$} \\ [15pt] 
\hline \\
$\U(3)\oplus \U(2)\oplus \U(2)$ & \specialcell[c]{$[2]_*(\LL_\psi(\lambda_1 u^{-1})) \oplus [2]_*(\LL_\psi(\lambda_2 u^{-1}))$ \\ $\oplus\,[2]_*(\LL_\psi((\lambda_1+\lambda_2) u^{-1}) \oplus (-\mathbbm{1})$} \\ [15pt]
\hline \\
$\iota \oplus \iota\oplus-\iota\oplus-\iota\oplus -\mathbbm{1}^2\oplus \mathbbm{1}$ & \specialcell[c]{$[3]_*(\LL_\psi(\lambda_1 u^{-1}))$ \\ $\oplus\,[3]_*(\LL_\psi(-\lambda_1 u^{-1}))\oplus\mathbbm{1}$} \\ [15pt]
\hline \\
$\U(7)$ & $[6]_*(\LL_\psi(\lambda_1 u^{-1}))\oplus-\mathbbm{1}$ \\ [10pt]
$\varepsilon\U(3)\oplus \varepsilon^{-1}\U(3)\oplus\mathbbm{1}$ & $[6]_*(\LL_\psi(\lambda_1 u^{-1}))\oplus-\mathbbm{1}$ \\ [10pt]
$z\U(2)\oplus z^{-1}\U(2)\oplus z^2\oplus z^{-2}\oplus\mathbbm{1}$ & $[6]_*(\LL_\psi(\lambda_1 u^{-1}))\oplus-\mathbbm{1}$ \\ [10pt]
$x\U(2)\oplus x^{-1}\U(2)\oplus \U(3)$ & $[6]_*(\LL_\psi(\lambda_1 u^{-1}))\oplus-\mathbbm{1}$ \\ [10pt]
$x\oplus y\oplus xy\oplus(xy)^{-1}\oplus y^{-1}\oplus x^{-1}\oplus \mathbbm{1}$ & $[6]_*(\LL_\psi(\lambda_1 u^{-1}))\oplus-\mathbbm{1}$ \\ [10pt]
\end{tabular}
\end{center}
Conversely, the above list exhausts all possible local monodromies of wildly ramified irreducible rigid $\ell$-adic local systems on open subsets of $\PP^1$ with monodromy group $G_2$ of slopes with numerator 1.
\end{thm}
\textbf{Acknowledgements.} The author would like to thank Michael
Dettweiler \& Stefan Reiter for their support during the work on this project. The author was funded by the DFG SPP 1489.


\section{Rigid Local Systems and the Katz-Arinkin Algorithm}\label{rigid}
For the rest of this article let $k$ be the algebraic closure of a finite field of characteristic $p$ and fix a prime $\ell\neq p$. Let $j:U\hookrightarrow \PP^1_k$ be a non-empty open subset with complement $S$. An $\ell$-adic local system $\LL$ can be given as a continuous representation
\[\rho: \piet(U,\overline{u})\rightarrow \GL_n(\Qlb)\]
of the étale fundamental group with $\Qlb$-coefficients. For any $x\in S$ we denote by $I_x$ the inertia group at $x$ and we say that $\rho$ is rigid if and only if the collection $\{[\rho |_{I_x}]\}_{x\in S}$ of isomorphism classes of continuous $I_x$-representations determines $\rho$ up to isomorphism. \\
Recall that the index of rigidity of an $\ell$-adic local system is given by 
\[\rig(\LL)=\chi(\PP^1,j_*\END(\LL)).\]
We call the local system $\LL$ cohomologically rigid if $\rig(\LL)=2$. 
\begin{prop}[\cite{Fu19},  Thm 0.9 \& \cite{Ka96}, Thm 5.0.2] An irreducible local system $\LL$ on $j:U\hookrightarrow \PP^1$ is rigid if and only if it is cohomologically rigid. 
\end{prop}
We would like to link the index of rigidity to invariants of the local monodromy in order to be able to compute it from knowledge of local information only. In order to do that we recall the local setting. Let $K=k(\!(t)\!)$ and $I$ its absolute Galois group, called the inertia group. We denote by $\Ktame$ the maximal tamely ramified extension of $K$ and by $P$ its absolute Galois group, which we will call the wild ramification subgroup. We have an exact sequence 
\[1\rightarrow P \rightarrow I\rightarrow \Itame \rightarrow 1\]
where the tame inertia $\Itame\cong \varprojlim_{(n,p)=1} \mu_n(k)$ is an inverse limit over $n$-th roots of unity in $k$ for $n$ prime to $p$. 
\begin{lem} \label{inertiasplit}
The sequence 
\[1\rightarrow P \rightarrow I\rightarrow \Itame \rightarrow 1\]
splits. In particular there is a subgroup $H\subset I$ isomorphic to $\Itame$.
\end{lem}
\begin{proof}
The group $\Itame$ is the maximal pro-$p'$ quotient of $I$ and $P$ is the pro-$p$-Sylow subgroup of $I$. Therefore the assertion follows from the profinite version of the Schur-Zassenhaus  Theorem \cite[Prop. 2.3.3.]{Wils98}.
\end{proof}
The wildness of the ramification can be measured by two kinds of invariants. They are the slopes (also called breaks) and the Swan conductor. 
\begin{thm}[Slope Decomposition, \cite{Ka88}, 1.1.] \label{slopedecomp}
Let $\rho:I\rightarrow \GL(V)$ be a continuous representation of $I$ with coefficients in $\Qlb$. There is a unique decomposition 
\[V=\bigoplus_{y\in \Q_{\ge 0}} V(y) \]
where only finitely many $V(y)$ do not vanish. These $y$ are called the $\textit{slopes}$ of $V$. The number $\Sw(V)=\sum_{y\in \Q_{\ge 0}} y\dim V(y)$ is called the Swan conductor of $V$ and is a non-negative integer. 
The representation $V$ is tame if and only if all of its slopes vanish or equivalently if $\Sw(V)=0$. 
\end{thm}
We can now compute the Euler characteristic of an $\ell$-adic local system by means of local information using the Euler-Poincaré formula.
\begin{prop}[\cite{Fu15}, Corollary 10.2.7] Let $\LL$ be an $\ell$-adic local system on an open subset $j:U\hookrightarrow \PP^1_k$ corresponding to the representation $\rho$ of $\piet(U,\overline{u})$, let $S=\PP^1_k-U$ and $s=\#S$. We have
\[\chi(\PP^1,j_*\LL)=(2-s)\rk(\LL)-\sum_{x\in S} \left(\Sw(\rho_x)-\dim(\rho_x)^{I_x}\right). \]
\end{prop}

For the rest of the article we fix a non-trivial additive character $\psi:\IF_p\rightarrow \Qlb^*$ and denote by $\LL_\psi$ the Artin-Schreier sheaf on $\A^1$ associated to the character $\psi$.

Let us briefly recall the definition of middle convolution with Kummer sheaves, cf. \cite[Chapter 2]{Ka96}. Denote by $\Perv(\A^1)$ the category of $\ell$-adic perverse sheaves on the affine line and let $K$ be a perverse sheaf on $\A^1$. Denote by $\LL_\chi$ the Kummer sheaf on $j:\Gm\hookrightarrow \A^1$ corresponding to the character $\chi: k^*\rightarrow \Qlb^*$ and by $m$ the addition map of $\A^1$. We have the two convolutions \[K*_! j_*\LL_\chi[1]=m_!(K\boxtimes j_*\LL_\chi[1])\] and 
\[K*_* j_*\LL_\chi[1]=m_*(K\boxtimes j_*\LL_\chi[1])\]
in the derived category $D^b_c(\A^1, \Qlb)$. There is a natural morphism 
\[K*_! j_*\LL_\chi[1]\rightarrow K*_*j_*\LL_\chi[1]\]
and we denote its image by $K*_\midd j_*\LL_\chi[1]$. We obtain the middle convolution functor
 \begin{align*}
MC_\chi:\Perv(\A^1)\rightarrow \Perv(\A^1) \\
K\mapsto K*_\midd j_*\LL_\chi[1].
\end{align*}
If $\G$ is an $\ell$-adic local system on $j:U\hookrightarrow \A^1$ and $K=j_*\G[1]$, the convolution $K*_\midd j_*\LL_\chi[1]$ will again be of the form $j_*\G'[1]$ for some $\ell$-adic local system $\G'$ on $U$. For ease of notation we will sometimes write $\MC_\chi(\G)=\G'$ in this situation. As mentioned before, the main theorem about the structure of tamely ramified local systems is the following.
\begin{thm}[\cite{Ka96}, Thm 5.2.1.]\label{tameranklow}
Let $\G$ be a tamely ramified cohomologically rigid $\ell$-adic local system on some non-empty proper open subset of $\A^1$ of rank at least $2$. Then there exists a tame $\ell$-adic local system $\LL$ of rank one and a character $\chi$ as above such that
\[\rk(\HH^{-1}(\MC_\chi(j_*(j^*\G\otimes j^*\LL)[1]))) < \rk(\G)\]
where $j:U\hookrightarrow \A^1$ is the embedding of an open subset of $\A^1$ where both $\G$ and $\LL$ are lisse and $\HH^i$ denotes the cohomology in degree $i$. 
\end{thm}
We wish to extend this theorem to include $\ell$-adic local systems which are wildly ramified. To that end we recall the definition of the Fourier transform for $\ell$-adic sheaves. Let $A=\A^1_t$ be the affine line with coordinate $t$ and dual $A'=\A^1_{t'}$ and denote by 
\[m:A\times_k A'\rightarrow \Ga \]
the canonical pairing. Let $\pr:A\times_k A'\rightarrow A$ and $\pr':A\times_k A'\rightarrow A'$ be the projections. The \textit{Fourier transform} with respect to the non-trivial character $\psi:k\rightarrow \Qlb$ is the functor
\[\F_\psi : \Perv(A,\Qlb)\rightarrow \Perv(A',\Qlb) \]
given by
\[\F_\psi(K)=\Rb \pr_!'(\pr^*K\otimes \LL_\psi(m))[1] \]
for $K$ an object in $\Perv(A)$.  One of the most important features of the Fourier transform in dimension $1$ is the principle of stationary phase. It will allow us to control the behvaviour of local monodromy after Fourier transform.  To state it, we introduce the following notation. Let $\eta_s$ be the formal punctured disc around $s$ and $\eta_{\infty'}$ be the formal punctured disc around $\infty'$ (i.e. in the coordinate after Fourier transform). We denote by $\F_{\psi}^{(0,\infty')}$ Laumon's local Fourier transform as defined in \cite[Def 2.4.2.3]{Laumon87}. 
\begin{prop}[\cite{Ka90}, Corollary 7.4.2]\label{statphase}
Let $k$ be the algebraic closure of a finite field, $j:U\hookrightarrow \A^1$ be an open subset, $S$ its complement, $\LL$ a lisse irreducible sheaf on $U$ and $K=j_*\LL[1]$ its middle extension. Furthermore let $K'=\F(K)$ and $\LL'=\mathscr{H}^{-1}(K'|_{U'})$ where $U'$ is the maximal open subset of $\A^1$ where $K'$ has lisse cohomology sheaves. We then have
\[\LL'|_{\eta_{\infty'}} =  \bigoplus_{s\in S} \left(\F_{\psi}^{(0,\infty')}(\LL|_{\eta_s}/\LL|_{\eta_s}^{I_s})\otimes  \LL_\psi(sx')\right) \oplus \F_\psi^{(\infty,\infty')}(\LL_{\eta_\infty}). \]
\end{prop}
The stationary phase formula also allows for the computation of the generic rank of the Fourier transform.
\begin{cor}\label{fourierrank}
Suppose that $k$ is algebraically closed. Let $j:U\hookrightarrow \A^1$ be an open subset, $\LL$ a lisse irreducible sheaf on $U$ and $K=j_*\LL[1]$ its middle extension. With notations as before the rank of $\LL'$ is then
\[\rk(\LL')=\sum_{s\in S} \left(\Sw(\LL|_{\overline{\eta}_s})+\rk(\LL)-\rk(\LL|_{\overline{\eta}_s}^{I_s})\right)+\Sw(\LL|^{>1}_{\eta_\infty})-\rk(\LL|_{\overline{\eta}_s}^{>1}).\]
\end{cor}
An analogue of Theorem \ref{tameranklow} holds under some hypotheses for local systems with not necessarily tame ramification if we make use of the Fourier transform. The following theorem is analogous to \cite[Thm A]{Arinkin10} and its proof is essentially the same.  
\begin{thm}\label{lowerrank}
Let $\LL$ be an irreducible rigid $\ell$-adic local system on $j:U\hookrightarrow \PP^1$ of $\rk(\LL)>1$ with slopes $\frac{k_1}{d_1},...,\frac{k_v}{d_v}$ all written in lowest terms. Assume that we have $\rk(\LL)<\textup{char}(k)=p$ and $\max\{k_1,...,k_r\}<p$. Then one of the following holds:
\begin{compactenum}[(i)]
\item There exists a tame character $\lambda:\piet(\Gm,1)\rightarrow \Qlb^*$ and an $\ell$-adic system $\chi$ of rank one on $U-\{\infty\}$ such that if we let $K=\MC_\lambda((j_*\HOM(\chi,\LL)[1])$, $V$ the open subset of $\PP^1$ where $\HH^{-1}(K)$ is lisse and $\MC_\lambda(\HOM(\chi,\LL)):=\HH^{-1}(K)|_V$ we have
\[\rk(\MC_\lambda(\HOM(\chi,\LL)))<\rk(\LL). \]
\item There is $\phi\in\Aut(\PP^1)$ and an $\ell$-adic local system $\chi$ of rank one on $U$ such that if we let $k:\phi^{-1}(U)\hookrightarrow \PP^1$ the embedding, $K=\F(k_*\phi^*(\HOM(\chi,\LL)[1]))$, $V$ the open subset of $\PP^1$ on which $\HH^{-1}(K)$ is lisse and let
\[\F(\phi^*\HOM(\chi,\LL)):=\HH^{-1}(K)|_V\] we have
\[\rk(\F(\phi^*\HOM(\chi,\LL)))<\rk(\LL).\]
\end{compactenum}

\end{thm}
%
The additional input that one needs to prove Theorem \ref{lowerrank} in this setting is the following Lemma which is stated in \cite{Arinkin10}, Lemma 6.1, without proof and without the additional assumption $n<p$. 
\begin{lem} \label{sameslopereps}
Let $V$ and $W$ be $\Qlb$-representations of $I$. Let $x=\frac{n}{d}\in \Q_{\ge 0}$ with $(p,d)=1$ and $n<p$. We have
\[\dim((V\otimes W)(x))\ge \dim V(x)\dim W(x)(1-1/d). \]
\end{lem}
\begin{proof} First note that in any case 
\[\dim((V\otimes W)(x)) \ge \dim((V(x)\otimes W(x)))(x)).\]
Hence we can assume that $V=V(x)$ and $W=W(x)$ and we can furthermore assume that they are irreducible. Write $x=n/d$ with $(n,d)=1$. Since $p$ does not divide $d$, $V$ and $W$ are induced from the unique open normal subgroup $I(d)$ of index $d$ from characters $\chi$ and $\rho$ of slope $n$, see [\cite{Ka88}, 1.14.]. Let us write 
\[V=\Ind_{I(d)}^I \chi, W=\Ind_{I(d)}^I \rho.\]
By \cite[Thm. 10.18]{Curt90} we have 
\[V\otimes W=\bigoplus_{g\in I/I(d)} \Ind_{I(d)}^I  \,^g\chi\otimes \rho,\]
where $\,^g\chi$ denotes the conjugate representation of $\chi$ and by abuse of notation by $g$ we mean a lift in $I$. Because $p<n$ by \cite[8.5.7.1]{Ka88} we have 
\[\chi \cong \LL_\psi(a_nt^n+\dots+a_1t)\otimes \K_1 \]
and similarly 
\[\rho \cong \LL_\psi(b_nt^n+\dots+b_1t)\otimes \K_2 \]
where for $\varphi(u)\in k(\!(u)\!)$ by $\LL_\psi(\varphi(u))$ we denote the pull-back corresponding to the covering of formal disks given by $\varphi$ and where $\K_1$ and $\K_2$ are tamely ramified representations. In this setting if we identify $I/I(d)\cong \mu_d(k)$ then \[\,^g (\LL_\psi(\varphi(u))) \cong \LL_\psi(\varphi(gu)).\]
An explicit computation shows that the slope of $\,^g\chi \otimes \rho$ can only be less than $n$ if $a_ng^n+b_n=0$ and this can happen at most for one $g$. From this it follows that
\begin{align*}\dim((V\otimes W)(x))&=\sum_{g\in I/I(d)} \dim \left( \Ind_{I(d)}^I  \,^g\chi\otimes \rho\right)(x) \ge d(d-1) \end{align*}
because at most one summand can vanish. Finally we have 
\[d(d-1)  = \dim V(x)\dim W(x)(1-\frac{1}{d})\]
proving the claim.
\end{proof}
Using this Lemma one can check that the results of Section 4.3. of \cite{Arinkin10} hold in the arithmetic setting. The rest of the proof of Theorem \ref{lowerrank} works completely analogous. \par
Most importantly, as a corollary we have the following version of the Katz-Arinkin algorithm for rigid irreducible local systems having slopes with numerator 1.
\begin{cor}
Let $\LL$ be a rigid irreducible $\ell$-adic local system on $U\xhookrightarrow{j} \PP^1$ such that $\rk(\LL)<p$ and all of its slopes have numerator 1. After a finite sequence of Fourier transforms, coordinate changes by automorphisms of $\PP^1$ and twists with rank one local systems the sheaf $\LL$ is reduced to a tamely ramified $\Qlb$-sheaf of rank one. 
\end{cor}
Using the principle of stationary phase from Proposition \ref{statphase}, to understand the behaviour of local monodromy under Fourier transform it is therefore enough to understand the local Fourier transform of representations of the inertia group $I$. For a certain type of representations we can explicitly compute these transforms. These are analogues of the formal connections which are called elementary in \cite{Sa08} and correspond to sheaves of the form
\[[r]_*(\LL_\psi\otimes \K) \]
for some integer $r$ prime to $p$, $\varphi\in t^{-1}k[t^{-1}]$ and $\K$ some tamely ramified sheaf. 
\begin{prop}[\cite{Fu10}, Thm 0.1]\label{locfourp}Let $\A^1=\Spec k[t]$ with $k$ algebraically closed, $\K$ an indecomposable tamely ramified $\ell$-adic local system on $\Gm$ and denote by $t'$ the Fourier transform variable. Let $\rho(t)=t^r$ and
\[\varphi(t)=\frac{a_{-s}}{t^s}+...+\frac{a_{-1}}{t}\in t^{-1}k[t^{-1}]\]
and let 
\[\widehat{\rho}(t)=-\frac{\frac{d}{dt}\varphi(t)}{\frac{d}{dt}\rho(t)},\ \widehat{\varphi}(t)=\varphi(t)+\rho(t)\widehat{\rho}(t). \]
Suppose that $2, r, s$ and $r+s$ are all prime to $p$ and denote by $\chi_2:\mu_2(k)\rightarrow \Qlb^*$ the unique quadratic character. We then have
\[\FI((\rho_*(\LL_\psi(\varphi(t))\otimes \K)|_{\eta_0})\cong \widehat{\rho}_*(\LL_\psi(\widehat{\varphi}(t))\otimes \K\otimes [s]^*\K_{\chi_2})|_{\eta_{\infty'}}.\]
\end{prop}
Even though an analogue of the Levelt-Turittin theorem does not hold in full generality we will see in the next section that in our setting it still suffices to understand representations of the above form. The construction of rigid local systems is then carried out exactly as in \cite{Ja20}. 

\section{Local Structure} \label{local}

A powerful tool for the classification in the complex setting is the Levelt-Turittin theorem. It describes the structure of $\KK$-connections in a very detailed way which allows us to explicitly compute the formal types of Fourier transforms. Under the right conditions we have the following weaker version of an analogue of the Levelt-Turrittin Theorem. 
\begin{thm}[\cite{Fu10} Prop. 0.5.]\label{poscharlevelt}
Let $\rho: I\rightarrow \GL(V)$ be an irreducible $\Qlb$-representation satisfying the following conditions. 
\begin{compactenum}[(i)]
\item Let $P$ be the wild inertia subgroup of $I$. Denote by $P^p$ $p$-th powers in $p$. Then $\rho(P^p[P,P])=1$.
\item The image $\rho(I)$ is finite. 
\item We have $s:=\Sw(\rho)<p$ where $\Sw(\rho)$ is the Swan conductor of $\rho$. 
\end{compactenum}
Then there is an integer $r$ not divisible by $p$, a tame character $\lambda$ of $I$ and a polynomial $\varphi\in u^{-1}k[u^{-1}]$ of degree $s$ such that 
\[V\cong \Ind_{I(r)}^I \left(\LL_\psi(\varphi) \otimes \lambda\right). \]
\end{thm}
Note that 
\[\textrm{Res}_{P}^I \textrm{Ind}_{I(r)}^I \left(\LL_\psi(\varphi(t))\otimes\lambda\right) \cong \bigoplus_{\zeta\in \mu_r(k)} \LL_\psi(\varphi(\zeta t)),\]
and this is a direct sum of characters factoring through $\mu_p(\Qlb)$. Hence it is trivial on $P^p[P,P]$. Therefore the first condition is a necessary condition for a representation to be of the desired shape. \par
Let $\zeta$ be a topological generator of $I^\textit{tame}$ and denote by $J$ the pre-image of $\zeta^\Z$ in $I$ under the canonical map $I\rightarrow I^\textit{tame}$. Then $J$ is a dense subgroup of $I$ and we have $J/P\cong \Z$ whose generator we also denote by $\zeta$. 
\begin{lem}[\cite{Fu10}, Lemma 2.2.]
Let $\rho: J\rightarrow \GL(V)$ be an irreducible representation over $\Qlb$. Then there is a character $\chi:J\rightarrow \Qlb^*$ trivial on $P$ such that $\rho\otimes \chi$ has finite image. 
\end{lem}
Regarding the second condition in Theorem \ref{poscharlevelt}, the following stronger statement holds. 
\begin{cor}\label{finiteimage}
Let $\rho :I\rightarrow \GL(V)$ be an irreducible $\Qlb$-representation of dimension $n$. Then there is a character $\chi:I\rightarrow \Qlb^*$ trivial on $P$ such that $\rho\otimes \chi$ has finite image. 
\end{cor}
\begin{proof}
Let $\tilde{\rho}=\rho|_J$ be the restriction of $\rho$ to $J$. This is again irreducible which can be seen as follows. Suppose it is not, then $\tilde{\rho}(J)$ stabilizes a subspace $W\subset V$ hence is contained in a proper parabolic subgroup $P$ of $GL(V)$. Since $\rho$ is continuous and $P$ is closed we have 
\[\rho(I)=\rho(\overline{J})\subset \overline{\tilde{\rho}(J)} \subset \overline{P}=P. \]
Therefore $\rho$ couldn't have been irreducible. We conclude that $\tilde{\rho}$ must be irreducible. By the above lemma there exists a character $\tilde{\chi}:J\rightarrow \Qlb^*$ such that $\tilde{\rho}  \otimes \tilde{\chi}$ has finite image in $\GL(V)$. Let $g\in J$ be an inverse image of $\zeta\in J/P$ and let $x=\tilde{\rho}\otimes\tilde{\chi}(g)$. The cyclic group generated by $x$ inside the image of $\tilde{\rho}\otimes\tilde{\chi}$ must be finite, so there is a positive integer $r$ such that $g^r$ lies in the kernel of $\tilde{\rho}\otimes\tilde{\chi}$. We find that
\[1=\det(\tilde{\rho}\otimes\tilde{\chi}(g^r))=\tilde{\chi}(g)^{rn} \det(\tilde{\rho}(g^r)). \]
Since $\rho(I)$ is compact, we can assume that it is a subgroup of $\GL_n(\OO_E)$ for a finite extension $E$ of $\Ql$. Now $\tilde{\rho}(g^r)=\rho(g^r)\in \GL_n(\OO_E)$ and $\tilde{\chi}(g)^{rn}\in \OO_E^*$. After a further finite extension $E\subset E'$ we get that $\tilde{\chi}$ factors through $\OO_{E'}^*$. The latter is compact, hence complete and we can extend 
$\tilde{\chi}:J\rightarrow \OO_{E'}^* $
by \cite[Page 96]{Hu82} to a character 
\[\chi:I\rightarrow \OO_{E'}^*\hookrightarrow \Qlb. \]
Finally we have
\[\rho\otimes\chi(I)=\rho\otimes\chi(\overline{J})\subset \overline{\tilde{\rho}\otimes\tilde{\chi}(J)}=\tilde{\rho}\otimes\tilde{\chi}(J) \]
proving the claim. 
\end{proof}
This means that the following stronger version of Theorem \ref{poscharlevelt} is true. 
\begin{cor}\label{LT-pos}
Let $\rho : I\rightarrow \GL(V)$ be an indecomposable $\Qlb$-representation. Suppose $\rho(P^p[P,P])=1$ and $\Sw(\rho)<p$. Then the lisse $\Qlb$-sheaf on $\eta=\Spec \,k(\!(t)\!)$ corresponding to $\rho$ is isomorphic to
\[[r]_*(\LL_\psi(\varphi)\otimes \K) \]
where $r$ is an integer prime to $p$, $[r](u)=u^r$, $\K$ is a tamely ramified $\Qlb$-sheaf on $\eta$, $\LL_\psi$ is the Artin-Schreier sheaf and $\varphi$ is a polynomial in $u^{-1}$ where $u^r=t$. 
\end{cor}
\begin{cor} Let $\rho : I\rightarrow \GL(V)$ be an indecomposable $\Qlb$-representation. Suppose $\rho(P^p[P,P])=1$ and $\Sw(\rho)<p$. Then the same is true for $\FI(V)$. 
\end{cor}
\begin{proof} By the corollary, $V\cong [r]_*(\LL_\psi(\varphi)\otimes \K)$ with $\deg(\varphi)=\Sw(\rho)$. Now by Theorem \ref{locfourp} the local Fourier transform $\FI(V)$ is of a similar shape with the same Swan conductor and hence satisfies the desired conditions. 
\end{proof}
In particular, we obtain a Levelt-Turrittin-type decomposition for the local monodromy of rigid local systems with slopes having numerator 1. Note that the tame sheaf $\K$ can be given in terms of a Jordan form. Denote by $\U(n)$ the representation of $\Itame$ given by mapping the topological generator to a Jordan block of length $n$. Then any indecomposable representation of $\Itame$ of rank $n$ can be written as $\chi\otimes \U(n)$ for $\chi$ some character.

\section{Classification} \label{classif}
To carry out the same classification as in \cite{Ja20} we need the following tools:
\begin{compactenum}
\item A way to compute the determinant of representations of the form \[[r]_*(\LL_\psi(\varphi)\otimes \K)\]
\item and tensor products of such objects,
\item an analogue of formal monodromy (see \cite[Section 1]{Mitschi96}), giving us constraints on the tame sheaves $\K$
\item and an analogue of the exponential torus, providing constraints on the $\varphi$. 
\end{compactenum}
We will discuss these in the given order. \par
\begin{prop}
The determinant of the representation $\rho$ associated to \[[r]_*(\LL_\psi(\varphi(u))\otimes \K)\]
with $(r,p)=1$ is given by
\[\det(\rho)=(\chi_2)^{(r-1)n}\cdot \chi_{n\textup{Tr}\, \varphi(t)}\otimes \det(\K) \]
where $n$ is the rank of $\K$, $\chi_{n\textup{Tr}\, \varphi(t)}$ is the character associated to $\LL_\psi(n\textup{Tr}\,\varphi(t))$ and $\textup{Tr}\,\varphi(t)$ is the trace of $\varphi(u)$ with respect to the Galois extension $k(\!(t)\!)\subset k(\!(u)\!)$. 
\end{prop}
\begin{proof}
The representation $\rho$ is induced from the unique normal subgroup $I(r)$ of $I$. Using the projection formula we reduce to the case $[r]_*\LL_\psi(\varphi(u))$. Denote by $\chi$ the character corresponding to $\LL_\psi(\varphi(u))$. 
By \cite[Prop. 13.15.]{Curt90} we have
\[\det\Ind_{I(r)}^I (\chi)=\varepsilon_{I\rightarrow I(r)}\cdot(\chi\circ V_{I(r)}^I) \]
where $\varepsilon_{I\rightarrow I(r)}(\sigma)$ is the sign of the permutation induced by $\sigma$ on $I/I(r)$ and $V_{I(r)}^I$ is the transfer map. We refer to \cite[13.10]{Curt90} for the definition of the transfer map. To compute the character 
\[\varepsilon_{I\rightarrow I(r)}:I\rightarrow \Qlb^* \]
first note that since $I(r)$ is normal the permutation representation $\pi:I\rightarrow S_r, \sigma\mapsto \pi_\sigma$ on $I/I(r)$ factors through $I/I(r)\cong \mu_r(k)$.
We therefore have the following commutative diagram
\[\xymatrix{I \ar[rr]^{\varepsilon_{I\rightarrow I(r)}}\ar[dr] & & \Qlb^* \\
 & \mu_r(k) \ar[ur] & 
 }\]
and we denote the map $ \mu_r(k)\rightarrow \Qlb^*$ also by $\varepsilon_{I\rightarrow I(r)}$. Choose representatives $g_i$ of $I/I(r)$ for $i=0,...,r-1$ in such a way that the image of $g_i$ in $\mu_r(k)$ is $\zeta_r^{i}$ where $\zeta_r$ is a primitve $r$-th root of unity. In this case the permutation associated to $g_i$ is $\pi_i(j)=j+i \mod r$. Now $\varepsilon_{I\rightarrow I(r)}(g_1)=\textup{sgn}(\pi_1)=(-1)^{r-1}$. We can view $\varepsilon_{I\rightarrow I(r)}$ as a map $\Itame\rightarrow \Qlb^*$ and we see that $\varepsilon_{I\rightarrow I(r)}(\zeta)=(-1)^{r-1}$ where $\zeta$ denotes the topological generator of $\Itame$. Hence $\varepsilon_{I\rightarrow I(r)}=\chi_2^{r-1}$ where $\chi_2$ is the unique quadratic character. It remains to compute $\phi:=\chi\circ V_{I(r)}^I:I\rightarrow \Qlb^*$. Note that for $\sigma\in I(r)$ we have $\chi_{\varphi(u)}(g_i^{-1}\sigma g_i)=\chi_{\varphi(\zeta_r^iu)}(\sigma)$. By the definition of transfer 
\[V_{I(r)}^I(\sigma)=\prod_{i=0}^{r-1} g_{\pi_\sigma(i)}^{-1}\sigma g_i.\]
Recall that the sequence 
\[1\rightarrow P \rightarrow I\rightarrow \Itame \rightarrow 1 \]
splits by the profinite Schur-Zassenhaus theorem and that we have a subgroup $H\subset I$ which is isomorphic to $\Itame$ such that $I=PH$ and $H\cap P =1$. Let $\sigma\in H$. We have $\sigma=\tau^p$ for some $\tau$ as every element in $H$ is a $p$-th power. Therefore we find 
\[\phi(\sigma)=\chi(V_{I(r)}^I(\sigma))=\chi((V_{I(r)}^I(\tau))^p)=1. \]
For a general element $\sigma \in I$ we have $\sigma=\sigma_P\sigma_H$ with $\sigma_P\in P$ and $\sigma_H\in H$. Since we have $P\subset I(r)$ and the Artin-Schreier character $\LL_\psi(\textup{Tr} \varphi(u))$ is also trivial on $H$ we compute
\[\phi(\sigma)=\phi(\sigma_P)\phi(\sigma_H)=\chi\left(\prod_{i=0}^{r-1} g_i^{-1}\sigma_P g_i\right)=\LL_\psi(\textup{Tr}\varphi(u))(\sigma_P)=\LL_\psi(\textup{Tr}\varphi(u))(\sigma). \]
Here we used the additivity
\[\bigotimes_{i=0}^{r-1} \LL_\psi(\varphi(\zeta_r^i u))\cong \LL_\psi(\textup{Tr}\varphi(u))\]
of the Artin-Schreier sheaf. We have therefore computed both factors of the determinant, proving the claim.
\end{proof}
\begin{cor}
Suppose that in the situation of the above proposition $s<r$. The sheaf
\[\det([r]_*(\LL_\psi(\varphi(u))\otimes \K))\]
is tamely ramified. 
\end{cor}
\begin{proof}
It is enough to prove the claim for $\varphi(u)=a_{-s}/u^s$. We have 
\[\textup{Tr}(\varphi(u))=a_{-s}\sum_{\zeta\in\mu_r(k)}(\zeta^s)^{-1}\frac{1}{u^s}.\]
The map
\[\mu_r(k)\rightarrow \mu_r(k), \zeta\mapsto \zeta^s \]
defines a non-trivial charater of $\mu_r(k)$, hence $\sum_{\zeta\in\mu_r(k)}(\zeta^s)^{-1}=0$. Therefore $\textup{Tr}(\varphi(u))=0$ and the sheaf is tamely ramified. 
\end{proof}
Proposition \cite[Prop. 3.8.]{Sa08} provides a detailed formula to compute tensor products of elementary connections $[r]_*(\E^\varphi\otimes R)$. A similar formula is true in our setting.
\begin{prop} \label{tensorinduced}
Let $\rho_i(u)=u^{r_i}$, $d=\gcd(r_1,r_2)$, $r_i'=r_i/d$, $\rho_i'(u)=u^{r_i'}$ and $\rho(u)=u^{\frac{r_1r_2}{d}}$. Suppose that $p$ does not divide either $r_1$ or $r_2$. For two polynomials $\varphi_1, \varphi_2\in \frac{1}{t}k[\frac{1}{t}]$ we set $\varphi^{(k)}(u)=\varphi_1(u^{r_2'})+\varphi_2((\zeta_{r_1r_2/d}^{k}u)^{r_1'})$ where $\zeta_{r_1r_2/d}$ is a primitive $\frac{r_1r_2}{d}$-th root of unity. In addition let $\K_1$ and $\K_2$ be tamely ramified $\ell$-adic local systems on $\eta$ and let $\K=(\rho_2')^*\K_1\otimes (\rho_1')^*\K_2$. We then have
\[\rho_{1,*}(\LL_\psi(\varphi_1(u))\otimes \K_1)\otimes \rho_{2,*}(\LL_\psi(\varphi_2(u))\otimes \K_2)\cong \bigoplus_{k=0}^{d-1} \rho_*(\LL_\psi(\varphi^{(k)}(u))\otimes \K). \]
\end{prop}
\begin{proof}
The proof is an application of Mackey theory. First notice that because of the projection formula we can reduce to the case of $\K_1=\K_2=\Qlb$. We regard all the sheaves as representations of respective Galois groups 
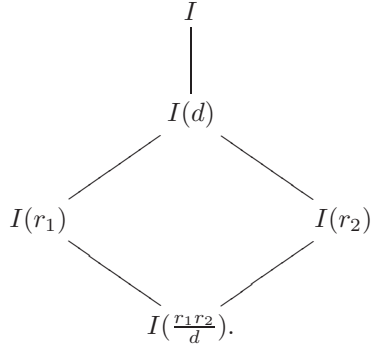
\[\xymatrix{& I\ar@{-}[d] & \\
& I(d)\ar@{-}[dl]\ar@{-}[dr] & \\
I(r_1)\ar@{-}[dr] & & I(r_2)\ar@{-}[dl] \\
& I(\frac{r_1r_2}{d}). &
}\]
In this language we have to compute the tensor product of induced representations
\[V:=\Ind_{I(r_1)}^I \LL_\psi(\varphi_1) \otimes\Ind_{I(r_2)}^I \LL_\psi(\varphi_2). \]
We have $I(r_1)\cdot I(r_2)= I(d)$ and $I(r_1)\cap I(r_2)= I(\frac{r_1r_2}{d})$. In addition all these subgroups are normal, hence stable under conjugation and furthermore we have
\[I(r_1)\backslash I/I(r_2)\cong I(r_1)I(r_2)\backslash I\cong \mu_d(k).\]
We apply \cite[Thm. 10.18]{Curt90} for to obtain 
\[V\cong \bigoplus_{i=0}^{d-1}\Ind_{I(\frac{r_1r_2}{d})}^I\Big(\Res_{I(\frac{r_1r_2}{d})}^{I(r_1)}\LL_\psi(\varphi_1) \otimes \Res_{I(\frac{r_1r_2}{d})}^{I(r_2)}\LL_\psi(\varphi_2 \circ m_{\zeta^k})\Big) \]
where $m_\zeta(u)=\zeta u$ for a primitive $\frac{r_1r_2}{d}$-th root of unity $\zeta$. The representation
\[\Res_{I(\frac{r_1r_2}{d})}^{I(r_1)}\LL_\psi(\varphi_1) \otimes \Res_{I(\frac{r_1r_2}{d})}^{I(r_2)}\LL_\psi(\varphi_2\circ m_{\zeta^k})\]
is isomorphic to
\[\LL_\psi(\varphi_1\circ \rho_2')\otimes \LL_\psi(\varphi_2\circ\mu_{\zeta^k}\circ\rho_1')\cong \LL_\psi(\varphi^{(k)}),\]
hence translating back to sheaves yields the claim. 
\end{proof}
Consider the sheaf $[r]_*(\LL_\psi(\varphi(t))\otimes \K)$ where $r$ is a positive integer prime to $p$, $\K$ is an indecomposable tamely ramified sheaf and denote by $\rho$ its associated representation. Recall that by Lemma \ref{inertiasplit} for the wild inertia group $P$ of $I$ we have the exact sequence
\[1\rightarrow P\rightarrow I\rightarrow I^{\textrm{tame}}\rightarrow 1\]
where $P$ is the pro-$p$-Sylow subgroup and $\Itame$ is the maximal prime-to-$p$-quotient of $I$. In particular there is a subgroup $H\subset I$ such that $H\cong \Itame$ and $I\cong P\rtimes \Itame$. Recall that after a choice $\K^{1/r}$ of an $r$-th root of $\K$ we have 
\[[r]_*(\LL_\psi(\varphi(t))\otimes \K)\cong [r]_*(\LL_\psi(\varphi(t)))\otimes \K^{1/r}. \]
We want to compute 
\[\Res_H^I \Ind_{I(r)}^I \LL_\psi(\varphi(t)) \]
to obtain the tame monodromy of the induced Artin-Schreier sheaf. By the Mackey Subgroup Theorem \cite[Thm. 10.13]{Curt90} we have
\[\Res_H^I \Ind_{I(r)}^I \LL_\psi(\varphi(t))  \cong \bigoplus_{x\in I/I(r)H} \Ind_{I(r)\cap H}^H\Res_{I(r)\cap H}^{I(r)} \,^x \LL_\psi(\varphi(t)).\]
One can check that $I(r)\cap H=H(r)$ where $H(r)$ is the corresponding subgroup obtained through the Schur-Zassenhaus theorem for $I(r)$. 
Since $\LL_\psi(\varphi(t))$ is trivial on $p$-th powers in $I(r)$ and every element of $H(r)$ is a $p$-th power, 
\[\Res_{H(r)}^{I(r)} \,^x \LL_\psi(\varphi(t))=\one \]
is the trivial representation. Therefore 
\[\Res_H^I \Ind_{I(r)}^I \LL_\psi(\varphi(t))=\Res_H^I \Ind_{I(r)}^I \one.\]
As a representation of $H\cong \Itame$ the representation $\Ind_{I(r)}^I \one$ maps the topological generator to the cyclic permutation matrix $\mathbb{P}_r$ of dimension $r$. Restricting the representation $\rho$ corresponding to 
\[[r]_*(\LL_\psi(\varphi(t)))\otimes \K^{1/r}\]
to $H$ therefore yields the tame sheaf $\K^{1/r}\otimes \mathbb{P}_r $. This is the analogue of formal monodromy in differential Galois theory. \par
The exponential torus is a diagonal subgroup of the differential Galois group coming from the relations satisfied by the exponential factors of formal solutions to a $\KK$-connection, see \cite[Section 11.22.]{Zoladek06}. 
Denote by $\rho$ the representation $\Ind_{I(r)}^I(\LL_\psi(\varphi(u))\otimes \lambda)$
where $\lambda$ is a tamely ramified character of $I$. By the projection formula we have 
\[\Ind_{I(r)}^I(\LL_\psi(\varphi(u))\otimes \lambda) \cong \Ind_{I(r)}^I(\LL_\psi(\varphi(u)))\otimes \lambda^{1/r} \]
for any choice of $r$-th root of $\lambda$. Restricting the representation $\rho$ to the wild ramification subgroup $P\subset I(r)$ yields the diagonal shape
\[\rho|_{P}\cong \bigoplus_{\zeta\in \mu_r(k)} \LL_\psi(\varphi(\zeta t)).\]
In particular the image $T:=\rho(P)$ is a diagonal subgroup of the monodromy group. Noting that 
\[\LL_\psi(\varphi(t))\otimes \LL_\psi(\beta(t))=\LL_\psi(\varphi(t)+\beta(t))\]
 we obtain the same relations for the $\varphi(\zeta t)$ as in the differential setting. \par
The exponential torus provided a method to analyze of what form the exponential factors in the differential setting could be. This will almost carry over to this setting. The only instance where it does not is \cite[Lemma 5.3.]{Ja20} whose proof we have to modify. 
\begin{lem}
Let $\LL$ be an irreducible rigid $\ell$-adic local system with monodromy group $G_2$ on some open subset of $\PP^1$ with all slopes having numerator 1 and let $V_x$ be its local monodromy at some singularity $x$ of $\LL$. The pole order of any $\varphi$ appearing in the analogue of the Levelt-Turrittin decomposition of $V_x$ can only be $1$ or $2$.
\end{lem}
\begin{proof}
We have the following table of possible cases for the ramification order $r$ and for the pole order $s$.
\begin{center}
\begin{tabular}{c c c}
$s$ & $r$  \\ [3pt]
\hline  \\
$2$ & $2,4,6$ \\ [5pt]
$3$ & $3,6$ \\ [5pt]
$4$ & $4$ \\ [5pt]
$6$ & $6$   \\ [5pt]
\end{tabular}
\end{center}
All cases apart from $s=3$ and $r=6$ or $r=3$ are excluded in the same way as in the proof of \cite[Lemma 5.3.]{Ja20}. We will deal with these two remaining cases separately. Let us consider the case $s=3$ and $r=3$. The local monodromy of $V_x$ then contains a module of the form 
\[\Ind_{I(3)}^I (\LL_\psi(\varphi(u))\otimes \lambda) \]
where $\lambda$ is a tame character and 
\[\varphi(u)=a_3u^{-3}+a_2u^{-2}+a_1u^{-1}\]
with $a_3\neq 0$. This representation is not self-dual and therefore its dual also has to appear. This means that 
\[V_x\cong \Ind_{I(3)}^I (\LL_\psi(\varphi(u))\otimes \lambda) \oplus\Ind_{I(3)}^I (\LL_\psi(-\varphi(u))\otimes \lambda^\vee) \oplus \lambda' \]
for some tame character $\lambda'$. Denote by $\rho_x$ the homomorphism corresponding to $V_x$. A general element in $\rho_x(P_x(3))$ is of the form
\[(x,y,z,x^{-1},y^{-1},z^{-1},1).\]
To prove that there are elements not contained in $G_2(\Qlb)$ it is therefore enough to show that there is no relation $xy=z$, $xz=y$ or $yz=x$. This can be reformulated as follows. Let $\zeta_3$ be a primitive $3$-rd root of unity. We have to show that there is no relation 
\[\varphi(u)+\varphi(\zeta_3u)=\varphi(\zeta_3^2u)\]
and the other combinations respectively. Note that the coefficient of $u^{-3}$ in $\varphi(\zeta_3^iu)$ is the same for all $i$. Therefore any of these relations translates into $a_3+a_3=a_3$. Since $s=3$ we have $a_3\neq 0$ and hence there cannot be a relation of the above form. \par
The case $s=3$ and $r=6$ is similar. We consider a representation of the form
\[\Ind_{I(6)}^I \LL_\psi(\varphi(u))\otimes \lambda\]
with  $\varphi(u)=a_3u^{-3}+a_2u^{-2}+a_1u^{-1}$. This representation has to be self-dual which in turn forces $a_2=0$. In this case 
\[V_x\cong \Ind_{I(6)}^I (\LL_\psi(\varphi(u))\otimes \lambda) \oplus \lambda' \]
for  a tame character $\lambda'$. Let $\zeta_6$ be a primitive $6$-th root of unity. We have the following relations
\begin{align*}
\varphi(u)+\varphi(\zeta_6^3u)=0, \\
\varphi(\zeta_6 u)+\varphi(\zeta_6^4u)=0, \\
\varphi(\zeta_6^2u)+\varphi(\zeta_6^5u)=0.
\end{align*}
Therefore elements in $\rho_x(P_x(6))$ are of the form
\[(x,y,z,x^{-1},y^{-1},z^{-1},1).\]
As before we have to show that there are no relations $xy=z$, $xz=y$ or $yz=x$. In terms of the leading coefficient of $\varphi(\zeta_6^iu)$ for $i=1,2,3$ this translates into 
$a_3-a_3=a_3$, $a_3+a_3=-a_3$ and $-a_3+a_3=a_3$ respectively. Because the characteristic $p>7$ in all cases from these relations it would follow that $a_3=0$. But we have $a_3\neq 0$ because $s=3$. Therefore none of these relations are satisfied and we find elements in $\rho_x(P_x)$ which do not lie in $G_2(\Qlb)$.
\end{proof}
Theorem \ref{classloc} is now obtained by the following methods. The index of rigidity yields constraints on Swan conductors and dimensions of invariants of the local monodromy by using the results on tensor products and determinants. We obtain further constraints on the shape of the local monodromy by means of the analogues of exponential torus and formal monodromy. A case-by-case check of the remaining possibilities of the Levelt-Turittin-type decomposition of the local monodromy yields the classification theorem in the arithmetic setting. For a detailed proof we refer to \cite[Section 6]{Ja20}. The arguments are completely analogous after replacing all objects by their respective counterparts. 

\bibliographystyle{jk}
\bibliography{mybib}

\providecommand{\arxivref}[1]{\href{http://arxiv.org/abs/#1}{#1}}
\providecommand{\bysame}{\leavevmode\hbox to3em{\hrulefill}\thinspace}
\providecommand{\MR}{\relax\ifhmode\unskip\space\fi MR }
\providecommand{\MRhref}[2]{%
  \href{http://www.ams.org/mathscinet-getitem?mr=#1}{#2}
}
\providecommand{\href}[2]{#2}
\begin{thebibliography}{Ka3}

\bibitem[Ar]{Arinkin10}
\emph{D.~Arinkin}\emph{: }Rigid irregular connections on {$\mathbb{P}^1$}.
  Compos. Math. \textbf{146} (2010), 1323--1338.

\bibitem[BE]{BlochEsnault04}
\emph{S.~Bloch}, \emph{H.~Esnault}\emph{: }Local {F}ourier transforms and
  rigidity for {$\mathcal{D}$}-modules. Asian J. Math. \textbf{8} (2004),
  587--605.

\bibitem[CR]{Curt90}
\emph{C.~W. Curtis}, \emph{I.~Reiner}\emph{: }Methods of representation theory.
  {V}ol. {I}. John Wiley \& Sons, Inc., New York, 1990, With applications to
  finite groups and orders, Reprint of the 1981 original, A Wiley-Interscience
  Publication.

\bibitem[DR]{Dett10}
\emph{M.~Dettweiler}, \emph{S.~Reiter}\emph{: }Rigid local systems and motives
  of type {$G_2$}. Compos. Math. \textbf{146} (2010), 929--963, With an
  appendix by Michael Dettweiler and Nicholas M. Katz. \MR{2660679
  (2011g:14042)}

\bibitem[Fu1]{Fu10}
\emph{L.~Fu}\emph{: }Calculation of {$\ell$}-adic local {F}ourier
  transformations. Manuscripta Math. \textbf{133} (2010), 409--464.

\bibitem[Fu2]{Fu15}
\emph{L.~Fu}\emph{: }Etale cohomology theory. Revised ed.. Nankai Tracts in
  Mathematics \textbf{14}. World Scientific Publishing Co. Pte. Ltd.,
  Hackensack, NJ, 2015.

\bibitem[Fu3]{Fu19}
\emph{L.~Fu}\emph{: }Deformations and rigidity of {$\ell$}-adic sheaves. Adv.
  Math. \textbf{351} (2019), 947--966. \MR{3956001}

\bibitem[Hu]{Hu82}
\emph{J.~E. Humphreys}\emph{: }Arithmetic groups. Topics in the theory of
  algebraic groups. Notre Dame Math. Lectures \textbf{10}. Univ. Notre Dame
  Press, South Bend, Ind.-London, 1982, 73--99.

\bibitem[Ja]{Ja20}
\emph{K.~Jakob}\emph{: }Classification of rigid irregular {$G_2$}-connections.
  Proceedings of the London Mathematical Society \textbf{120} (2020), 831--852.

\bibitem[Ka1]{Ka88}
\emph{N.~M. Katz}\emph{: }Gauss sums, {K}loosterman sums, and monodromy groups.
  Annals of Mathematics Studies \textbf{116}. Princeton University Press,
  Princeton, NJ, 1988.

\bibitem[Ka2]{Ka90}
\emph{N.~M. Katz}\emph{: }Exponential sums and differential equations. Annals
  of Mathematics Studies \textbf{124}. Princeton University Press, Princeton,
  NJ, 1990.

\bibitem[Ka3]{Ka96}
\emph{N.~M. Katz}\emph{: }Rigid local systems. Annals of Mathematics Studies
  \textbf{139}. Princeton University Press, Princeton, NJ, 1996.

\bibitem[La]{Laumon87}
\emph{G.~Laumon}\emph{: }Transformation de {F}ourier, constantes d'\'equations
  fonctionnelles et conjecture de {W}eil. Inst. Hautes \'Etudes Sci. Publ.
  Math. (1987), 131--210.

\bibitem[Mi]{Mitschi96}
\emph{C.~Mitschi}\emph{: }Differential {G}alois groups of confluent generalized
  hypergeometric equations: an approach using {S}tokes multipliers. Pacific J.
  Math. \textbf{176} (1996), 365--405.

\bibitem[Sa]{Sa08}
\emph{C.~Sabbah}\emph{: }An explicit stationary phase formula for the local
  formal {F}ourier-{L}aplace transform. Singularities {I}. Contemp. Math.
  \textbf{474}. Amer. Math. Soc., Providence, RI, 2008, 309--330.

\bibitem[Wi]{Wils98}
\emph{J.~S. Wilson}\emph{: }Profinite groups. London Mathematical Society
  Monographs. New Series \textbf{19}. The Clarendon Press, Oxford University
  Press, New York, 1998.

\bibitem[Zo]{Zoladek06}
\emph{H.~Zoladek}\emph{: }The monodromy group. Instytut Matematyczny Polskiej
  Akademii Nauk. Monografie Matematyczne (New Series) [Mathematics Institute of
  the Polish Academy of Sciences. Mathematical Monographs (New Series)]
  \textbf{67}. Birkh\"auser Verlag, Basel, 2006.

\end{thebibliography}

\end{document}